\documentclass[12pt]{amsart}
\usepackage{amssymb}
\usepackage{hyperref}
\usepackage[all]{xy}

\newtheorem{theorem}{Theorem}[section]
\newtheorem{definition}[theorem]{Definition}
\newtheorem{proposition}[theorem]{Proposition}
\newtheorem{lemma}[theorem]{Lemma}

\begin{document}

\title{Random walk questions for linear quantum groups}

\author{Teodor Banica}
\address{T.B.: Department of Mathematics, Cergy-Pontoise University, 95000 Cergy-Pontoise, France. {\tt teodor.banica@u-cergy.fr}}

\author{Julien Bichon}
\address{J.B.: Department of Mathematics, Clermont-Ferrand University, F-63177 Aubiere, France. {\tt bichon@math.univ-bpclermont.fr}}


\begin{abstract}
We study the discrete quantum groups $\Gamma$ whose group algebra has an inner faithful representation of type $\pi:C^*(\Gamma)\to M_K(\mathbb C)$. Such a representation can be thought of as coming from an embedding $\Gamma\subset U_K$. Our main result, concerning a certain class of examples of such quantum groups, is an asymptotic convergence theorem for the random walk on $\Gamma$. The proof uses various algebraic and probabilistic techniques.
\end{abstract}

\maketitle

\section*{Introduction}

A discrete quantum group $\Gamma$ is dual to a compact quantum group $G$, and vice versa. Of particular interest is the case where $\Gamma$ is finitely generated, which corresponds to the case where $G$ is a matrix quantum group. The associated unital Hopf $C^*$-algebras $A=C^*(\Gamma)=C(G)$ were axiomatized by Woronowicz in \cite{wo1}, \cite{wo2}.

The algebra $A$ has a Haar functional, $\int:A\to\mathbb C$, and possesses a certain distinguished element $\chi=Tr(u)\in A$, and one interesting problem is that of computing the probabilistic distribution of $\chi\in(A,\int)$. There are two motivations for this problem:
\begin{enumerate}
\item Random walks. For a discrete group $\Gamma=<g_1,\ldots,g_N>$ we have $\chi=g_1+\ldots+g_N$, whose moments are the numbers $c_p=\#\{i_1,\ldots,i_p|g_{i_1}\ldots g_{i_p}=1\}$.

\item Representation theory. For a compact group $G\subset_u U_N$ we have $\chi(g)=Tr(u(g))$, whose moments are the numbers $c_p=\dim(Fix(u^{\otimes p}))$.
\end{enumerate}

In general, understanding the structure of $A$, and computing the law of $\chi=Tr(u)$, are non-trivial questions. The available methods here fall into two main classes:
\begin{enumerate}
\item Category theory. When the relations between the standard coordinates $u_{ij}\in A$ are known, by definition or computation, Tannakian methods apply.

\item Matrix models. Here the idea is to search for models for the variables $u_{ij}\in A$. Once such a model found, matrix analysis gives us information about $\chi$.
\end{enumerate}

The first idea is well-known, going back to old work of Brauer \cite{bra}, when $G$ is classical, and to old work of Kesten \cite{kes}, when $\Gamma$ is classical. In the quantum group context, this idea has been heavily developed, starting with \cite{wo2}. See \cite{bsp}, \cite{fre}, \cite{rwe}.

The second idea, while having a big potential as well, is more of an ``underground'' one. Only some general algebraic theory is available here (\cite{abi}, \cite{ban}, \cite{bb2}, \cite{bfs}, \cite{chi}, \cite{end}), and there is still a lot of work to be done, in order for this idea to really ``take off''.

We will do here some work in this direction. First of all, a matrix model for $A$ will be by definition an inner faithful representation of type $\pi:A\to M_K(\mathbb C)$. The existence of such a representation is a linearity type condition on $\Gamma$, because in the classical case, the representation must come from a group embedding $\Gamma\subset U_K$. 

The simplest models are those coming from the Fourier representations $Z\subset U_{|Z|}$ of the finite abelian groups $Z$. Assuming now that we are in a product situation, $Z=X\times Y$, the corresponding matrix model can be twisted by a parameter belonging to a torus, $Q\in\mathbb T^Z$, and produces in this way a certain quantum group $G_Q$. We will study here $G_Q$, and the random walk on the dual quantum group $\Gamma_Q$, our main result being:

\medskip

\noindent {\bf Theorem.} {\em When $Q\in\mathbb T^Z$ is generic, with $|X|=\alpha K,|Y|=\beta K$, $K\to\infty$ we have
$$law(\chi)=\left(1-\frac{1}{\alpha\beta K^2}\right)\delta_0+\frac{1}{\alpha\beta K^2}D_{\frac{1}{\beta K}}(\pi_{\alpha/\beta})$$
where $\pi_t$ is the Marchenko-Pastur (or free Poisson) law of parameter $t$.}

\medskip

The proof uses various algebraic and probabilistic techniques, notably Hopf algebra methods from \cite{ade}, \cite{bb2}, \cite{bic}, \cite{sch} and free probability theory from \cite{bia}, \cite{mpa}, \cite{nsp}, \cite{vdn}.

Generally speaking, the situation that we have, with quantum groups $G_Q$  which are undeformed ($S^2=id$) and which depend critically on the arithmetics of $Q\in\mathbb T^Z$, is of course quite beautiful, and reminds a bit the theory of the algebras $U_q(\mathfrak g)$ at $|q|=1$, coming from \cite{dri}, \cite{jim}. At $|Z|=4$ the arithmetic specializations can be computed by using methods from \cite{bb1}. In general, this remains to be explored.

From a more applied point of view now, the main raison d'\^etre of the compact quantum groups is that of acting on noncommutative manifolds coming from quantum physics, cf. \cite{bdd}, \cite{con}. There are many interesting questions here, regarding the potential applications of the matrix model techniques. One problem is that of unifying the present results with those in \cite{bco}, and then trying to investigate more specialized models.

The paper is organized as follows: 1 is a preliminary section, in 2-3-4 we develop a number of algebraic methods for computing the quantum groups associated to the above matrix models, and in 5 we state and prove the random walk results.

\medskip

\noindent {\bf Acknowledgements.} The work of TB was partly supported by the ``Harmonia'' NCN grant 2012/06/M/ST1/00169. 

\section{Matrix models}

We fix a Hopf algebra $A=C^*(\Gamma)=C(G)$, satisfying Woronowicz's axioms in \cite{wo1}, \cite{wo2}. We assume in addition that the square of the antipode is the identity, $S^2=id$. By \cite{wo1}, this is the same as assuming that the Haar integration functional $\int:A\to\mathbb C$ has the trace property $\int ab=\int ba$. In short, we use the ``minimal'' framework covering the finitely generated groups, $\Gamma=<g_1,\ldots,g_N>$, and the compact Lie groups, $G\subset U_N$.

The axioms are in fact very simple, as follows:

\begin{definition}
A unitary Hopf algebra is a pair $(A,u)$ formed by a $C^*$-algebra $A$, and a unitary matrix $u\in M_N(A)$ whose transpose $u^t$ is unitary too, such that:
\begin{enumerate}
\item The formula $\Delta(u_{ij})=\sum_ku_{ik}\otimes u_{kj}$ defines a morphism $\Delta:A\to A\otimes A$.

\item The formula $\varepsilon(u_{ij})=\delta_{ij}$ defines a morphism $\varepsilon:A\to\mathbb C$.

\item The formula $S(u_{ij})=u_{ji}^*$ defines a morphism $S:A\to A^{op}$.
\end{enumerate}
We write $A=C^*(\Gamma)=C(G)$, and call $\Gamma,G$ the underlying quantum groups.
\end{definition}

At the level of basic examples, given a finitely generated group $\Gamma=<g_1,\ldots,g_N>$, we have the group algebra $A=C^*(\Gamma)$, with $u=diag(g_1,\ldots,g_N)$. Also, given a compact Lie group $G\subset U_N$, we have the algebra $A=C(G)$, with $u_{ij}(g)=g_{ij}$. See \cite{wo1}, \cite{wo2}.

We refer to the recent book \cite{ntu} for a detailed presentation of the theory.

Let us explain now what we mean by matrix model for $A$. We can of course consider embeddings of type $A\subset M_K(\mathbb C)$, with $K<\infty$, but this is not very interesting, because it can only cover the finite dimensional case. Observe that such an embedding always exists at $K=\infty$, by the GNS theorem, but this is just a theoretical result.

The answer comes from the notion of inner faithfulness, introduced in \cite{bb2}:

\begin{definition}
Let $\pi:A\to R$ be a $C^*$-algebra representation. 
\begin{enumerate}
\item The Hopf image of $\pi$ is the smallest quotient Hopf algebra $A\to A'$ producing a factorization of type $\pi:A\to A'\to R$.

\item When $A=A'$, we say that $\pi$ is inner faithful. That is, we call $\pi:A\to R$ inner faithful when there is no factorization $\pi:A\to A'\to R$.
\end{enumerate}
\end{definition}

Here the existence of $A'$ as in (1) comes from standard Hopf algebra theory. See \cite{bb2}.

As a basic example, when $\Gamma$ is a classical group, the representation $\pi$ must come from a unitary group representation $\Gamma\to U_R$, and the factorization in (1) is simply the one obtained by taking the image, $\Gamma\to\Gamma'\subset U_R$. Thus $\pi$ is inner faithful when $\Gamma\subset U_R$.

Also, given a compact group $G$, and elements $g_1,\ldots,g_K\in G$, we can consider the representation $\pi=\oplus_iev_{g_i}:C(G)\to\mathbb C^K$. The minimal factorization of $\pi$ is then via $C(G')$, with $G'=\overline{<g_1,\ldots,g_K>}$. Thus $\pi$ is inner faithful when $G=\overline{<g_1,\ldots,g_K>}$.

Finally, observe that the representation $A'\to R$ constructed in Definition 1.2 (1) is inner faithful. Thus, we have many other potential examples. See \cite{bb2}, \cite{chi}.

Now back to the matrix model problematics, we can formulate:

\begin{definition}
A matrix model for $A$ is a $C^*$-algebra representation
$$\pi:A\to M_K(\mathbb C)$$
which is inner faithful in the sense of Definition 1.2.
\end{definition}

When the underlying discrete quantum group $\Gamma$ is classical, such a model must come from a group embedding $\Gamma\subset U_K$. At the group dual level, given a compact group $G$, and elements $g_1,\ldots,g_K\in G$, we can consider the representation $\pi:C(G)\to M_K(\mathbb C)$ given by $\pi(\varphi)=diag(\varphi(g_i))$. By the above, $\pi$ is a matrix model when $G=\overline{<g_1,\ldots,g_K>}$.

Further examples include the fibers of the Pauli matrix representation of $A=C(S_4^+)$, studied in \cite{bb1}, \cite{bco}, \cite{end}. When dropping the assumption $S^2=id$, examples appear as well from certain $q$-deformations of enveloping Lie algebras, with $|q|\neq1$, see \cite{abi}. Let us also mention that, given an abstract algebra $A$ satisfying the axioms in Definition 1.1, deciding whether $A$ has or not a matrix model is a subtle analytic problem. See \cite{chi}.

Let us record, for future reference, the group and group dual statements:

\begin{proposition}
Given $g_1,\ldots,g_K\in U_N$, consider the discrete group $\Gamma=<g_1,\ldots,g_K>$, and the compact group $G=\overline{\Gamma}$. We have then matrix models, as follows:
\begin{enumerate}
\item $\pi:C^*(\Gamma)\to M_N(\mathbb C)$, mapping $g\to g$.

\item $\nu:C(G)\to M_K(\mathbb C)$, mapping $\varphi\to diag(\varphi(g_i))$.
\end{enumerate}
\end{proposition}

\begin{proof}
Both the assertions are elementary, and follow from the above discussion. For full details here, we refer to our previous paper \cite{bb2}.
\end{proof}

As explained in the introduction, we are interested here in using matrix models for solving some concrete questions, regarding the random walk on $\Gamma$. In order to discuss such questions, we must first study the Haar functional of $A$. We recall that such a Haar functional exists, thanks to the general results of Woronowicz in \cite{wo1}. 

We use in what follows multi-indices of exponents, $\varepsilon=(\varepsilon_1,\ldots,\varepsilon_p)\in\{1,*\}^p$. Given a square matrix $w\in M_n(R)$, we define $w^{\otimes\varepsilon}\in M_{n^p}(R)$ by $w^\varepsilon=w_{1,p+1}^{\varepsilon_1}\ldots w_{p,p+1}^{\varepsilon_p}$, using the leg-numbering notation, and the standard identification $M_{n^p}(R)\simeq M_n(\mathbb C)^{\otimes p}\otimes R$.

The general available results on the matrix models can be summarized as follows:

\begin{proposition}
Let $\pi:A\to M_K(\mathbb C)$ be a matrix model, mapping $u_{ij}\to U_{ij}$.
\begin{enumerate}
\item We have $Fix(u^{\otimes\varepsilon})=Fix(U^{\otimes\varepsilon})$, where $Fix(W)=\{\xi|W\xi=\xi\}$.

\item $(\int_Gu_{a_1b_1}^{\varepsilon_1}\ldots u_{a_pb_p}^{\varepsilon_p})_{a_1\ldots a_p,b_1\ldots b_p}$ is the orthogonal projection on $Fix(U^{\otimes\varepsilon})$.

\item $\int_G=\lim_{k\to\infty}\frac{1}{k}\sum_{r=1}^k\int_G^r$, where $\int_G^r=(tr\circ\pi)^{*r}$, with $\phi*\psi=(\phi\otimes\psi)\Delta$.

\item $\int_G^ru_{a_1b_1}^{\varepsilon_1}\ldots u_{a_pb_p}^{\varepsilon_p}=(T_\varepsilon^r)_{a_1\ldots a_p,b_1\ldots b_p}$, where $(T_\varepsilon)_{i_1\ldots i_p,j_1\ldots j_p}=tr(U_{i_1j_1}^{\varepsilon_1}\ldots U_{i_pj_p}^{\varepsilon_p})$.
\end{enumerate}
\end{proposition}

\begin{proof}
These results are known from \cite{bb2}, \cite{bfs}, the proof being as follows:

(1) This follows from Tannakian duality \cite{wo2}, see \cite{bb2}.

(2) This follows from (1) and from the Peter-Weyl theory in \cite{wo1}, see \cite{bb2}.

(3) This follows by using idempotent state methods, see \cite{bfs}.

(4) This formula, useful in conjunction with (3), is elementary, see \cite{ban}, \cite{bfs}.
\end{proof}

Let us try now to compute the Kesten type measure $\mu=law(Tr(u))$.

As a first observation, in the real case, $u=\bar{u}$, the character $\chi=Tr(u)$ is self-adjoint, and by unitarity of $u$, it satisfies $||\chi||\leq N$. Thus in this case $\mu$ is a probability measure, supported on $[-N,N]$. It is well-known that $\Gamma$ is amenable when $N\in supp(\mu)$.

In general, $\mu$ is a $*$-distribution, in the sense of noncommutative probability theory. Such a $*$-distribution is uniquely determined by its $*$-moments. See \cite{nsp}, \cite{vdn}.

We have the following result, coming from Proposition 1.5 above:

\begin{proposition}
Let $\mu^r$ be the law of $\chi=Tr(u)$ with respect to $\int_G^r=(tr\circ\pi)^{*r}$.
\begin{enumerate}
\item We have the convergence formula $\mu=\lim_{k\to\infty}\frac{1}{k}\sum_{r=0}^k\mu^r$, in moments.

\item The $*$-moments of $\mu^r$ are $c_\varepsilon^r=Tr(T_\varepsilon^r)$, where $(T_\varepsilon)_{i_1\ldots i_p,j_1\ldots j_p}=tr(U_{i_1j_1}^{\varepsilon_1}\ldots U_{i_pj_p}^{\varepsilon_p})$.
\end{enumerate}
\end{proposition}

\begin{proof}
These results are basically known since \cite{ban}, the proof being as follows:

(1) This follows from the limiting formula in Proposition 1.5 (3).

(2) This follows from Proposition 1.5 (4), by summing over $a_i=b_i$.
\end{proof}

The above discussion regarding $\mu$ applies to each $\mu^r$. More precisely, when $u=\bar{u}$ each $\mu^r$ is a probability measure on $[-N,N]$. In general, $\mu^r$ is a $*$-distribution.

Let us prove now that, under suitable assumptions, $\mu^r$ is the law of a certain explicit matrix. In order to do so, we will need a certain duality operation, as follows:

\begin{definition}
Let $\pi:A\to M_K(\mathbb C)$ be a matrix model, mapping $u_{ij}\to U_{ij}$.
\begin{enumerate}
\item We set $(U'_{kl})_{ij}=(U_{ij})_{kl}$, and define $\widetilde{\rho}:C(U_K^+)\to M_N(\mathbb C)$ by $v_{kl}\to U_{kl}'$.

\item We perform the Hopf image construction, as to get a model $\rho:A'\to M_N(\mathbb C)$.
\end{enumerate}
\end{definition}

Here the quantum group $U_K^+$ is the free analogue of the unitary group $U_K$, constructed by Wang in \cite{wa1}. More precisely, the algebra $C(U_K^+)$ is by definition the universal one generated by the entries of a $K\times K$ biunitary matrix $v=(v_{ij})$. See \cite{wa1}.

Observe that the matrix constructed in (1) is given by $U'=\Sigma U$, where $\Sigma$ is the flip. Thus this matrix is indeed biunitary, and produces a representation $\rho$ as in (1).

The operation $A\to A'$ is a duality, in the sense that we have $A''=A$. Let us first discuss a few basic examples, following \cite{ban}, \cite{bb2}. We use the duality operation $G\to\widehat{G}$ for the finite quantum groups, for which we refer to \cite{wo1}, and the basic group/group dual matrix model constructions, from Proposition 1.4 above:

\begin{proposition}
With $A=C(G),A'=C(G')$, we have:
\begin{enumerate}
\item $|G|<\infty\implies G'=\widehat{G}$.

\item $G=\overline{<g_1,\ldots,g_K>}\iff G'=\widehat{<g_1,\ldots,g_K>}$.
\end{enumerate}
\end{proposition}

\begin{proof}
These results are known since \cite{ban}, \cite{bb2}, the proof being as follows:

(1) Assume that $(C(G),u)$ is as in Definition 1.1 above, with $|G|<\infty$, and that we have a matrix model $\pi:C(G)\to M_K(\mathbb C)$. We can then construct a fundamental corepresentation for $C(G)^*$, by the formula $w_{kl}(x)=(\pi(x))_{kl}$, and a $*$-representation $\rho:C(G)^*\to M_N(\mathbb C)$, by the formula $\rho(\varphi)=(\varphi(u_{kl}))_{kl}$. We have:
$$\rho(w_{ab})=(w_{ab}(v_{kl}))_{kl}=(\rho(v_{kl})_{ab})_{kl}=((U_{kl})_{ab})_{kl}=((U'_{ab})_{kl})_{kl}=U'_{ab}$$

Thus we have $G'\subset\widehat{G}$, and by interchanging $G,\widehat{G}$, we obtain the result. See \cite{ban}.

(2) Given unitaries $g_1,\ldots,g_k\in U_N$, set $\Gamma=<g_1,\ldots,g_K>$ and $G=\overline{\Gamma}$. The standard examples of matrix models, described in Proposition 1.4 above, are as follows:

-- We have a model $\pi:C^*(\Gamma)\to M_N(\mathbb C)$, given by $g\to g$. Since we have $U_{ij}=\delta_{ij}g_i$ in this case, the associated biunitary matrix is diagonal, $U=\sum_le_{ll}\otimes g_l$.

-- We have a model $\nu:C(G)\to M_K(\mathbb C)$, given by $\nu(\varphi)=diag(\varphi(g_l))$. Here we have $U_{ij}=diag((g_l)_{ij})=\sum_le_{ll}(g_l)_{ij}$, and so $U=\sum_{ijl}e_{ij}\otimes e_{ll}(g_l)_{ij}=\sum_lg_l\otimes e_{ll}$.

Summarizing, for these two models the respective biunitary matrices are $\sum_le_{ll}\otimes g_l$ and $\sum_lg_l\otimes e_{ll}$, related indeed by the flip operation $\Sigma$. See \cite{bb2} for details.
\end{proof}

We denote by $D$ the dilation operation for probability measures, or for general $*$-distributions, given by the formula $D_r(law(X))=law(rX)$. 

We have the following result, extending previous findings from \cite{ban}:

\begin{theorem}
Consider the rescaled measure $\eta^r=D_{1/N}(\mu^r)$.
\begin{enumerate}
\item The moments $\gamma_p^r=c_p^r/N^p$ of $\eta^r$ satisfy $\gamma_p^r(A)=\gamma_r^p(A')$.

\item $\eta^r$ has the same moments as the matrix $T_r'=T_r(A')$.

\item In the real case $u=\bar{u}$ we have $\eta^r=law(T_r')$.
\end{enumerate}
\end{theorem}

\begin{proof}
All results follow from Proposition 1.6 (2), as follows:

(1) We have the following computation:
\begin{eqnarray*}
c_p^r(A)
&=&Tr(T_p^r)
=\sum_i(T_p)_{i_1^1\ldots i_p^1,i_1^2\ldots i_p^2}\ldots\ldots(T_p)_{i_1^r\ldots i_p^r,i_1^1\ldots i_p^1}\\
&=&\sum_itr(U_{i_1^1i_1^2}\ldots U_{i_p^1i_p^2})\ldots\ldots tr(U_{i_1^ri_1^1}\ldots U_{i_p^ri_p^1})\\
&=&\frac{1}{N^r}\sum_i\sum_j(U_{i_1^1i_1^2})_{j_1^1j_2^1}\ldots(U_{i_p^1i_p^2})_{j_p^1j_1^1}\ldots\ldots(U_{i_1^ri_1^1})_{j_1^rj_2^r}\ldots(U_{i_p^ri_p^1})_{j_p^rj_1^r}
\end{eqnarray*}

In terms of the matrix $(U'_{kl})_{ij}=(U_{ij})_{kl}$, then by permuting the terms in the product on the right, and finally with the changes $i_a^b\leftrightarrow i_b^a,j_a^b\leftrightarrow j_b^a$, we obtain:
\begin{eqnarray*}
c_p^r(A)
&=&\frac{1}{N^r}\sum_i\sum_j(U'_{j_1^1j_2^1})_{i_1^1i_1^2}\ldots(U'_{j_p^1j_1^1})_{i_p^1i_p^2}\ldots\ldots(U'_{j_1^rj_2^r})_{i_1^ri_1^1}\ldots(U'_{j_p^rj_1^r})_{i_p^ri_p^1}\\
&=&\frac{1}{N^r}\sum_i\sum_j(U'_{j_1^1j_2^1})_{i_1^1i_1^2}\ldots(U'_{j_1^rj_2^r})_{i_1^ri_1^1}\ldots\ldots(U'_{j_p^1j_1^1})_{i_p^1i_p^2}\ldots(U'_{j_p^rj_1^r})_{i_p^ri_p^1}\\
&=&\frac{1}{N^r}\sum_i\sum_j(U'_{j_1^1j_1^2})_{i_1^1i_2^1}\ldots(U'_{j_r^1j_r^2})_{i_r^1i_1^1}\ldots\ldots(U'_{j_1^pj_1^1})_{i_1^pi_2^p}\ldots(U'_{j_r^pj_r^1})_{i_r^pi_1^p}
\end{eqnarray*}

On the other hand, if we use again the above formula of $c_p^r(A)$, but this time for the matrix $U'$, and with the changes $r\leftrightarrow p$ and $i\leftrightarrow j$, we obtain:
$$c_r^p(A')=\frac{1}{N^p}\sum_i\sum_j(U'_{j_1^1j_1^2})_{i_1^1i_2^1}\ldots(U'_{j_r^1j_r^2})_{i_r^1i_1^1}\ldots\ldots(U'_{j_1^pj_1^1})_{i_1^pi_2^p}\ldots(U'_{j_r^pj_r^1})_{i_r^pi_1^p}$$

Now by comparing this with the previous formula, we obtain $N^rc_p^r(A)=N^pc_r^p(A')$. Thus we have $c_p^r(A)/N^p=c_r^p(A')/N^r$, and this gives the result.

(2) By using (1) and the formula in Proposition 1.6 (2), we obtain:
$$\frac{c_p^r(A)}{N^p}=\frac{c_r^p(A')}{N^r}=\frac{Tr((T'_r)^p)}{N^r}=tr((T'_r)^p)$$

But this gives the equality of moments in the statement.

(3) This follows from the moment equality in (2), and from the standard fact that for self-adjoint variables, the moments uniquely determine the distribution.
\end{proof}

\section{Projective models}

In general, the use of the above methods is quite limited. We restrict now attention to a certain special class of models, for which more general theory can be developed. 

We recall that a square matrix $u=(u_{ij})$ is called ``magic'' if its entries are projections ($p=p^2=p^*$), which sum up to 1 on each row and column. The basic example is provided by the matrix coordinates $u_{ij}:S_N\subset O_N\to\mathbb R$, given by $u_{ij}(\sigma)=\delta_{i\sigma(j)}$.

The following key definition is due to Wang \cite{wa2}:

\begin{definition}
$C(S_N^+)$ is the universal $C^*$-algebra generated by the entries of a $N\times N$ magic matrix $u$, with $\Delta(u_{ij})=\sum_ku_{ik}\otimes u_{kj}$, $\varepsilon(u_{ij})=\delta_{ij}$, $S(u_{ij})=u_{ji}$.
\end{definition}

This algebra satisfies the axioms in Definition 1.1, so the underlying space $S_N^+$ is a compact quantum group, called quantum permutation group. The canonical embedding $S_N\subset S_N^+$ is an isomorphism at $N=1,2,3$, but not at $N\geq4$. See \cite{wa2}.

Now back to the matrix models, we recall from Definition 1.7 that such models come in pairs, $\pi:A\to M_K(\mathbb C)$, $u_{ij}\to U_{ij}$ and $\pi':A'\to M_N(\mathbb C)$, $u_{ij}'\to U_{ij}'$, the connecting formula being $(U'_{ij})_{kl}=(U_{kl})_{ij}$, or, equivalently, $U'=\Sigma U$, where $\Sigma$ is the flip. We agree from now on to restrict the attention to the case $K=N$. We have:

\begin{definition}
A matrix model $\pi:A\to M_N(\mathbb C)$, with dual model $\pi':A'\to M_N(\mathbb C)$, is called projective if both $\pi,\pi'$ appear from representations of $C(S_N^+)$.
\end{definition}

In other words, with $A=C(G),A'=C(G')$, the projectivity condition states that we have $G,G'\subset S_N^+$. Equivalently, with $U=\sum_{ij}e_{ij}\otimes U_{ij}$ and $U'=\sum_{kl}e_{kl}\otimes U'_{kl}$, the projectivity condition states that both matrices $(U_{ij})$ and $(U'_{ij})$ must be magic.

The basic examples of such models are those coming from the complex Hadamard matrices. We recall that such a matrix, $H\in M_N(\mathbb C)$, has by definition its entries on the unit circle, and its rows are pairwise orthogonal. At the level of examples, the Fourier matrix $F_G$ of any finite abelian group $G$ is Hadamard, of size $N=|G|$. See \cite{tzy}.

The models associated to the Hadamard matrices are constructed as follows:

\begin{proposition}
If $H\in M_N(\mathbb C)$ is Hadamard, with rows $H_1,\ldots,H_N\in\mathbb T^N$, then
$$U_{ij}=Proj\left(\frac{H_i}{H_j}\right)=\frac{1}{N}\left(\frac{H_{ik}H_{jl}}{H_{il}H_{jk}}\right)_{kl}$$
produces a projective model. In addition we have $U'(H)=U(H^t)$.
\end{proposition}

\begin{proof}
The vectors $H_1,\ldots,H_N$ being pairwise orthogonal, we obtain:
$$\Big\langle\frac{H_i}{H_j},\frac{H_i}{H_k}\Big\rangle=\sum_r\frac{H_{ir}}{H_{jr}}\cdot\frac{H_{kr}}{H_{ir}}=\sum_r\frac{H_{kr}}{H_{jr}}=<H_k,H_j>=N\delta_{jk}$$

A similar computation gives $<H_i/H_j,H_k/H_j>=N\delta_{ik}$, so the matrix of rank one projections $U_{ij}=Proj(H_i/H_j)$ is magic. Moreover, by using the  formula $Proj(\xi)=\frac{1}{||\xi||^2}(\xi_i\overline{\xi}_j)_{ij}$, the projections $U_{ij}$ are indeed given by the formula in the statement.

Regarding now the last assertion, this follows from:
$$(U_{ij}')_{kl}=(U_{kl})_{ij}=\frac{1}{N}\cdot\frac{H_{ki}H_{lj}}{H_{kj}H_{li}}=\frac{1}{N}\cdot\frac{(H^t)_{ik}(H^t)_{jl}}{(H^t)_{jk}(H^t)_{il}}$$

In particular the matrix $U'$ is magic as well, and this finishes the proof.
\end{proof}

We can deform the tensor products of projective models, as follows:

\begin{proposition}
Given two projective models $\pi:A\to M_M(\mathbb C)$, $\nu:B\to M_N(\mathbb C)$, mapping $u_{ij}\to U_{ij},v_{ij}\to V_{ij}$, the matrix $W=U\otimes_QV$ given by
$$(W_{ia,jb})_{kc,ld}=\frac{Q_{ic}Q_{jd}}{Q_{id}Q_{jc}}(U_{ij})_{kl}(V_{ab})_{cd}$$
produces a projective model, for any choice of the parameter matrix $Q\in M_{M\times N}(\mathbb T)$.
\end{proposition}

\begin{proof}
Let us first check that the elements $W_{ia,jb}$ are self-adjoint. We have indeed:
\begin{eqnarray*}
(W_{ia,jb}^*)_{kc,ld}
&=&\overline{(W_{ia,jb})_{ld,kc}}=\frac{\bar{Q}_{id}\bar{Q}_{jc}}{\bar{Q}_{ic}\bar{Q}_{jd}}\overline{(U_{ij})_{lk}}\cdot\overline{(V_{ab})_{dc}}\\
&=&\frac{Q_{ic}Q_{jd}}{Q_{id}Q_{jc}}(U_{ij})_{kl}(V_{ab})_{cd}=(W_{ia,jb})_{kc,ld}
\end{eqnarray*}

We verify now the magic condition. First, we have:
\begin{eqnarray*}
(W_{ia,jb}W_{ia,me})_{kc,ld}
&=&\sum_{nf}(W_{ia,jb})_{kc,nf}(W_{ia,me})_{nf,ld}\\
&=&\sum_{nf}\frac{Q_{ic}Q_{jf}}{Q_{if}Q_{jc}}(U_{ij})_{kn}(V_{ab})_{cf}\frac{Q_{if}Q_{md}}{Q_{id}Q_{mf}}(U_{im})_{nl}(V_{ae})_{fd}\\
&=&\sum_f\frac{Q_{ic}Q_{jf}Q_{md}}{Q_{jc}Q_{id}Q_{mf}}(V_{ab})_{cf}(V_{ae})_{fd}\sum_n(U_{ij})_{kn}(U_{im})_{nl}
\end{eqnarray*}

The last sum on the right being $(U_{ij}U_{im})_{kl}=\delta_{jm}(U_{ij})_{kl}$, we obtain:
\begin{eqnarray*}
(W_{ia,jb}W_{ia,me})_{kc,ld}
&=&\delta_{jm}\sum_f\frac{Q_{ic}Q_{jd}}{Q_{jc}Q_{id}}(V_{ab})_{cf}(V_{ae})_{fd}(U_{ij})_{kl}\\
&=&\delta_{jm}\frac{Q_{ic}Q_{jd}}{Q_{jc}Q_{id}}(U_{ij})_{kl}\sum_f(V_{ab})_{cf}(V_{ae})_{fd}
\end{eqnarray*}

The last sum on the right being $(V_{ab}V_{ae})_{cd}=\delta_{be}(V_{ab})_{cd}$, we obtain:
$$(W_{ia,jb}W_{ia,me})_{kc,ld}=\delta_{jm}\delta_{be}\frac{Q_{ic}Q_{jd}}{Q_{jc}Q_{id}}(U_{ij})_{kl}(V_{ab})_{cd}=\delta_{jm}\delta_{be}(W_{ia,jb})_{kc,ld}$$

Thus the elements $W_{ia,jb}$ are indeed projections, which are pairwise orthogonal on rows. In order to conclude that $W$ is magic, we check that the sum on the columns is 1:
$$\sum_{ia}(W_{ia,jb})_{kc,ld}
=\sum_i\frac{Q_{ic}Q_{jd}}{Q_{id}Q_{jc}}(U_{ij})_{kl}\sum_a(V_{ab})_{cd}
=\delta_{cd}\sum_i(U_{ij})_{kl}=\delta_{cd}\delta_{kl}$$

It remains to prove that $W'$ is a projective model too. Since $W$ is a model, so is $W'$, so it suffices to prove that the entries of $W'$ are projections. We have:
$$(W'_{ia,jb})_{kc,ld}=(W_{kc,ld})_{ia,jb}=\frac{Q_{ka}Q_{lb}}{Q_{kb}Q_{la}}(U_{kl})_{ij}(V_{cd})_{ab}=\frac{Q_{ka}Q_{lb}}{Q_{kb}Q_{la}}(U'_{ij})_{kl}(V'_{ab})_{cd}$$

Now since both $U',V'$ are projective models, we obtain:
\begin{eqnarray*}
(W_{ia,jb}'^*)_{kc,ld}
&=&\overline{(W'_{ia,jb})_{ld,kc}}=\frac{\bar{Q}_{la}\bar{Q}_{kb}}{\bar{Q}_{lb}\bar{Q}_{ka}}\overline{(U'_{ij})_{lk}}\cdot\overline{(V'_{ab})_{dc}}\\
&=&\frac{Q_{lb}Q_{ka}}{Q_{la}Q_{kb}}(U'_{ij})_{kl}(V'_{ab})_{cd}=(W'_{ia,jb})_{kc,ld}
\end{eqnarray*}

Finally, the ckeck of the idempotent condition goes as follows:
\begin{eqnarray*}
(W_{ia,jb}'^2)_{kc,ld}
&=&\sum_{nf}(W'_{ia,jb})_{kc,nf}(W'_{ia,jb})_{nf,ld}\\
&=&\sum_{nf}\frac{Q_{ka}Q_{nb}}{Q_{kb}Q_{na}}(U'_{ij})_{kn}(V'_{ab})_{cf}\frac{Q_{na}Q_{lb}}{Q_{nb}Q_{la}}(U'_{ij})_{nl}(V'_{ab})_{fd}\\
&=&\frac{Q_{ka}Q_{lb}}{Q_{kb}Q_{la}}\sum_f(V'_{ab})_{cf}(V'_{ab})_{fd}\sum_n(U'_{ij})_{kn}(U'_{ij})_{nl}\\
&=&\frac{Q_{ka}Q_{lb}}{Q_{kb}Q_{la}}(V'_{ab})_{cd}(U'_{ij})_{kl}=(W'_{ia,jb})_{kc,ld}
\end{eqnarray*}

We conclude that $W'$ is magic too, and this finishes the proof.
\end{proof}

As an example, assume that we are given two Hadamard matrices, $H\in M_M(\mathbb C)$ and $K\in M_N(\mathbb C)$, and let us form the deformed tensor product $H\otimes_QK=(Q_{ib}H_{ij}K_{ab})_{ia,jb}$, which is Hadamard as well. See \cite{tzy}. The model associated to this matrix is:
$$(U_{ia,jb})_{kc,ld}=\frac{1}{MN}\cdot\frac{(Q_{ic}H_{ik}K_{ac})(Q_{jd}H_{jl}K_{bd})}{(Q_{id}H_{il}K_{ad})(Q_{jc}H_{jk}K_{bc})}=\frac{Q_{ic}Q_{jd}}{Q_{id}Q_{jc}}(U^H_{ij})_{kl}(U^K_{ab})_{cd}$$

Thus, the $\otimes_Q$ operations for models and for Hadamard matrices are compatible.

Another theoretical remark is that, in the projective model framework, the operation constructed in Proposition 2.4 has a dual counterpart, constructed as follows:

\begin{proposition}
With $U,V$ as in Proposition 2.4, the matrix
$W^\circ=U\!\!{\ }_Q\!\otimes V$ given by
$$(W^\circ_{ia,jb})_{kc,ld}=\frac{Q_{ka}Q_{lb}}{Q_{kb}Q_{la}}(U_{ij})_{kl}(V_{ab})_{cd}$$
produces a projective model. We have $W'=U'\!\!\!{\ }_Q\!\otimes V'$ and $W^{\circ\prime}=U'\otimes_QV'$.
\end{proposition}

\begin{proof}
We use the following formula, already met in the proof of Proposition 2.4:
$$(W'_{ia,jb})_{kc,ld}=(W_{kc,ld})_{ia,jb}=\frac{Q_{ka}Q_{lb}}{Q_{kb}Q_{la}}(U_{kl})_{ij}(U_{cd})_{ab}=
\frac{Q_{ka}Q_{lb}}{Q_{kb}Q_{la}}(U'_{ij})_{kl}(V'_{ab})_{cd}$$

With the convention in the statement for the products $\!\!{\ }_Q\otimes$, this means precisely that we have $W'=U'\!\!\!{\ }_Q\!\otimes V'$. The last assertion is proved similarly, because we have:
$$(W^{\circ\prime}_{ia,jb})_{kc,ld}=(W^\circ_{kc,ld})_{ia,jb}=\frac{Q_{ic}Q_{jd}}{Q_{id}Q_{jc}}(U_{kl})_{ij}(V_{cd})_{ab}=\frac{Q_{ic}Q_{jd}}{Q_{id}Q_{jc}}(U'_{ij})_{kl}(V'_{ab})_{cd}$$

Finally, the fact that the matrix $W^\circ$ produces indeed a projective model follows from Proposition 2.4, by using the two connecting formulae that we just proved.
\end{proof}

Once again, we have here a compatibility with the known complex Hadamard matrix constructions, and namely with the operation $H\!\!{\ }_Q\!\otimes K=(Q_{ja}H_{ij}K_{ab})_{ia,jb}$ from \cite{tzy}. Indeed, the projective model associated to such a matrix is:
$$(U_{ia,jb})_{kc,ld}=\frac{1}{MN}\cdot\frac{(Q_{ka}H_{ik}K_{ac})(Q_{lb}H_{jl}K_{bd})}{(Q_{la}H_{il}K_{ad})(Q_{kb}H_{jk}K_{bc})}=\frac{Q_{ka}Q_{lb}}{Q_{la}Q_{kb}}(U^H_{ij})_{kl}(U^K_{ab})_{cd}$$

As a last theoretical result about the deformed tensor products, constructed in Proposition 2.4 and Proposition 2.5 above, here is an alternative definition for them:

\begin{proposition}
We have $\widetilde{W}=\widetilde{U}_{13}Q^\delta\widetilde{V}_{24}$ and $\widetilde{W^\circ}=\widetilde{V}_{24}Q^\delta\widetilde{U}_{13}$, where 
$$Q^\delta=diag\left(\frac{Q_{ic}Q_{jd}}{Q_{id}Q_{jc}}\right)_{icjd}$$
and where the $U\to\widetilde{U}$ operation is defined by $(\widetilde{U}_{ij})_{kl}=(U_{jl})_{ik}$.
\end{proposition}

\begin{proof}
According to the definition of $W$ in Proposition 2.4, we have:
\begin{eqnarray*}
\widetilde{W}
&=&\sum_{iajb}\sum_{kcld}e_{ia,jb}\otimes e_{kc,ld}\frac{Q_{ja}Q_{lc}}{Q_{jc}Q_{la}}(U_{jl})_{ik}(V_{bd})_{ac}\\
&=&\sum_{iajb}\sum_{kcld}e_{ij}\otimes e_{ab}\otimes e_{kl}\otimes e_{cd}\frac{Q_{ja}Q_{lc}}{Q_{jc}Q_{la}}(\widetilde{U}_{ij})_{kl}(\widetilde{V}_{ab})_{cd}\\
&=&\widetilde{U}_{13}\left(\sum_{jalc}e_{jj}\otimes e_{aa}\otimes e_{ll}\otimes e_{cc}\frac{Q_{ja}Q_{lc}}{Q_{jc}Q_{la}}\right)\widetilde{V}_{24}
\end{eqnarray*}

We recognize in the middle the diagonal matrix $Q^\delta$ in the statement, and we are therefore done with the proof of the first formula. Similarly, we have:
\begin{eqnarray*}
\widetilde{W^\circ}
&=&\sum_{iajb}\sum_{kcld}e_{ia,jb}\otimes e_{kc,ld}\frac{Q_{ib}Q_{kd}}{Q_{kb}Q_{id}}(U_{jl})_{ik}(V_{bd})_{ac}\\
&=&\sum_{iajb}\sum_{kcld}e_{ij}\otimes e_{ab}\otimes e_{kl}\otimes e_{cd}\frac{Q_{ib}Q_{kd}}{Q_{kb}Q_{id}}(\widetilde{U}_{ij})_{kl}(\widetilde{V}_{ab})_{cd}\\
&=&\widetilde{V}_{24}\left(\sum_{ibkd}e_{ii}\otimes e_{bb}\otimes e_{kk}\otimes e_{dd}\frac{Q_{ib}Q_{kd}}{Q_{kb}Q_{id}}\right)\widetilde{U}_{13}
\end{eqnarray*}

But this gives the second formula in the statement, and we are done.
\end{proof}

We should mention the above description of the deformed tensor products is in fact not very enlightening, because the operation $U\to\widetilde{U}$ given by $(\widetilde{U}_{ij})_{kl}=(U_{jl})_{ik}$ does not map models to models, or projective models to projective models, even in the most simple cases. As an example here, for a Fourier matrix model, $(U_{ij})_{kl}=F_{i-j,k-l}$, where $F=F_X$ is the Fourier matrix of a finite abelian group $X$, we have $(\widetilde{U}_{ij})_{kl}=F_{j-l,i-k}$. Now since we have $(\widetilde{U}_{ij}^*)_{kl}=F_{j-k,l-i}$, we see that the matrices $\widetilde{U}_{ij}$ are not self-adjoint.

Let us study now the truncated moments. First, we have:

\begin{lemma}
The truncated moments for $W=U\otimes_QV$ are given by
$$c_p^r
=\frac{1}{(MN)^r}\sum_{ib}\Delta_U(i)\Delta_{V'}(b^t)\frac{Q_{i_1^1b_1^1}Q_{i_1^2b_2^1}}{Q_{i_1^1b_2^1}Q_{i_1^2b_1^1}}\ldots\frac{Q_{i_p^1b_p^1}Q_{i_p^2b_1^1}}{Q_{i_p^1b_1^1}Q_{i_p^2b_p^1}}\ldots\ldots\frac{Q_{i_1^rb_1^r}Q_{i_1^1b_2^r}}{Q_{i_1^rb_2^r}Q_{i_1^1b_1^r}}\ldots\frac{Q_{i_p^rb_p^r}Q_{i_p^1b_1^r}}{Q_{i_p^rb_1^r}Q_{i_p^1b_p^r}}$$
where $\Delta_U(i)=M^r\cdot(T_p^U)_{i_1^1\ldots i_p^1,i_1^2\ldots i_p^2}\ldots\ldots(T_p^U)_{i_1^r\ldots i_p^r,i_1^1\ldots i_p^1}$, for $i\in M_{r\times p}(1,\ldots,M)$.
\end{lemma}

\begin{proof}
We will use several times, in forward and in backwards form, the following computation, which already appeared in the proof of Theorem 1.9 (1) above:
\begin{eqnarray*}
c_p^r(U)
&=&\sum_i(T_p^U)_{i_1^1\ldots i_p^1,i_1^2\ldots i_p^2}\ldots\ldots(T_p^U)_{i_1^r\ldots i_p^r,i_1^1\ldots i_p^1}\\
&=&\sum_itr(U_{i_1^1i_1^2}\ldots U_{i_p^1i_p^2})\ldots\ldots tr(U_{i_1^ri_1^1}\ldots U_{i_p^ri_p^1})\\
&=&\frac{1}{M^r}\sum_i\sum_j(U_{i_1^1i_1^2})_{j_1^1j_2^1}\ldots(U_{i_p^1i_p^2})_{j_p^1j_1^1}\ldots\ldots(U_{i_1^ri_1^1})_{j_1^rj_2^r}\ldots(U_{i_p^ri_p^1})_{j_p^rj_1^r}
\end{eqnarray*}

In double index notation, with $U_{ij}$ replaced by $W_{ia,jb}$, the formula is:
\begin{eqnarray*}
c_p^r(W)
&=&\frac{1}{(MN)^r}\sum_{ia}\sum_{jb}(W_{i_1^1a_1^1,i_1^2a_1^2})_{j_1^1b_1^1,j_2^1b_2^1}\ldots\ldots(W_{i_p^1a_p^1,i_p^2a_p^2})_{j_p^1b_p^1,j_1^1b_1^1}\\
&&\hskip28mm\ldots\ldots\\
&&\hskip28mm(W_{i_1^ra_1^r,i_1^1a_1^1})_{j_1^rb_1^r,j_2^rb_2^r}\ldots\ldots(W_{i_p^ra_p^r,i_p^1a_p^1})_{j_p^rb_p^r,j_1^rb_1^r}
\end{eqnarray*}

Now with $W_{ia,jb}$ being as in Proposition 2.4 above, we obtain:
\begin{eqnarray*}
c_p^r(W)
&=&\frac{1}{(MN)^r}\sum_{ib}\frac{Q_{i_1^1b_1^1}Q_{i_1^2b_2^1}}{Q_{i_1^1b_2^1}Q_{i_1^2b_1^1}}\ldots\frac{Q_{i_p^1b_p^1}Q_{i_p^2b_1^1}}{Q_{i_p^1b_1^1}Q_{i_p^2b_p^1}}\ldots\ldots\frac{Q_{i_1^rb_1^r}Q_{i_1^1b_2^r}}{Q_{i_1^rb_2^r}Q_{i_1^1b_1^r}}\ldots\frac{Q_{i_p^rb_p^r}Q_{i_p^1b_1^r}}{Q_{i_p^rb_1^r}Q_{i_p^1b_p^r}}\\
&&\hskip15mm\sum_j(U_{i_1^1i_1^2})_{j_1^1j_2^1}\ldots(U_{i_p^1i_p^2})_{j_p^1j_1^1}\ldots\ldots(U_{i_1^ri_1^1})_{j_1^rj_2^r}\ldots(U_{i_p^ri_p^1})_{j_p^rj_1^r}\\
&&\hskip15mm\sum_a(V_{a_1^1a_1^2})_{b_1^1b_2^1}\ldots(V_{a_p^1a_p^2})_{b_p^1b_1^1}\ldots\ldots(V_{a_1^ra_1^1})_{b_1^rb_2^r}\ldots(V_{a_p^ra_p^1})_{b_p^rb_1^r}
\end{eqnarray*}

The middle sum can be compacted by using the computation in the beginning of this proof. The last sum can be compacted too, by using a similar computation, after switching indices by using $(V_{ab})_{cd}=(V'_{cd})_{ab}$. We obtain the following formula:
\begin{eqnarray*}
c_p^r(W)
&=&\frac{1}{(MN)^r}\sum_{ib}\frac{Q_{i_1^1b_1^1}Q_{i_1^2b_2^1}}{Q_{i_1^1b_2^1}Q_{i_1^2b_1^1}}\ldots\frac{Q_{i_p^1b_p^1}Q_{i_p^2b_1^1}}{Q_{i_p^1b_1^1}Q_{i_p^2b_p^1}}\ldots\ldots\frac{Q_{i_1^rb_1^r}Q_{i_1^1b_2^r}}{Q_{i_1^rb_2^r}Q_{i_1^1b_1^r}}\ldots\frac{Q_{i_p^rb_p^r}Q_{i_p^1b_1^r}}{Q_{i_p^rb_1^r}Q_{i_p^1b_p^r}}\\
&&\hskip22mm M^r\cdot (T_p^U)_{i_1^1\ldots i_p^1,i_1^2\ldots i_p^2}\ldots\ldots(T_p^U)_{i_1^r\ldots i_p^r,i_1^1\ldots i_p^1}\\
&&\hskip22mm N^p\cdot (T_r^{V'})_{b_1^1\ldots b_1^r,b_2^1\ldots b_2^r}\ldots\ldots(T_r^{V'})_{b_p^1\ldots b_p^r,b_1^1\ldots b_1^r}
\end{eqnarray*}

But this gives the formula in the statement, and we are done.
\end{proof}

In order to further advance, we use the following notion:

\begin{definition}
A model $\pi:A\to M_N(\mathbb C)$, mapping $u_{ij}\to U_{ij}$, is called positive if
$$tr(U_{i_1j_1}\ldots U_{i_pj_p})\geq0$$
for any $p\in\mathbb N$, and any choice of the indices $i_1,\ldots,i_p$ and $j_1,\ldots,j_p$.
\end{definition}

In other words, the model is called positive if the functional $\int_G^1=tr\circ\pi$ from Proposition 1.5 (3) is positive on all the products of standard coordinates $u_{i_1j_1}\ldots u_{i_pj_p}$. Equivalently, the matrix $T_p^U$ in Proposition 1.5 (4) must have positive entries, for any $p\in\mathbb N$.

Once again, the basic examples here come from the Hadamard matrices. In the context of Proposition 2.3 above, with the notations there, we have:
\begin{eqnarray*}
tr(U_{i_1j_1}\ldots U_{i_pj_p})
&=&\frac{1}{N}\sum_k(U_{i_1j_1})_{k_1k_2}\ldots\ldots(U_{i_pj_p})_{k_pk_1}\\
&=&\frac{1}{N^{p+1}}\sum_k\frac{H_{i_1k_1}H_{j_1k_2}}{H_{i_1k_2}H_{j_1k_1}}\ldots\ldots\frac{H_{i_pk_p}H_{j_pk_1}}{H_{i_pk_1}H_{j_pk_p}}\\
&=&\frac{1}{N^{p+1}}\sum_{k_1}\frac{H_{i_1k_1}H_{j_pk_1}}{H_{j_1k_1}H_{i_pk_1}}\ldots\ldots\sum_{k_p}\frac{H_{i_pk_p}H_{j_{p-1}k_p}}{H_{j_pk_p}H_{i_{p-1}k_p}}\\
&=&\frac{1}{N^{p+1}}\left\langle\frac{H_{i_1}}{H_{j_1}},\frac{H_{i_p}}{H_{j_p}}\right\rangle\ldots\ldots\left\langle\frac{H_{i_p}}{H_{j_p}},\frac{H_{i_{p-1}}}{H_{j_{p-1}}}\right\rangle
\end{eqnarray*}

In particular, if the quantities $C_{abcd}=\frac{1}{N}<\frac{H_a}{H_b},\frac{H_c}{H_d}>$ are all positive, then the positivity condition is satisfied. Observe that this is the case for the Fourier matrix $F_X$ of a finite abelian group $X$, where we have $C_{abcd}=\delta_{a-b,c-d}$, with all indices taken in $X$.

Now back to the general case, we have the following result, that we believe of interest, and which is the best one that we could find at the abstract level:

\begin{theorem}
If $U,V'$ come from positive projective models, with $W=U\otimes_QV$ we have:
$$|c_p^r(W)|\leq c_p^r(U)c_p^r(V)$$
Thus, the moments of $\mu^r_W$ are bounded by those of the usual tensor product.
\end{theorem}

\begin{proof}
By using $|Q_{ij}|=1$ for any $i,j$, the formula in Lemma 2.7 gives:
$$|c_p^r(W)|\leq\frac{1}{(MN)^r}\sum_{ib}|\Delta_U(i)\Delta_{V'}(b^t)|$$

Now observe that the computation in the beginning of the proof of Lemma 2.7 reads $\sum_i\Delta_U(i)=M^rc_p^r(U)$. Thus, assuming that we have positivity, the $\Delta$ quantities on the right are both positive, we can remove the absolute value sign, and we obtain:
\begin{eqnarray*}
|c_p^r(W)|
&\leq&\frac{1}{(MN)^r}\sum_{ib}\Delta_U(i)\Delta_{V'}(b^t)
=\frac{1}{(MN)^r}\cdot M^rc_p^r(U)\cdot N^pc_r^p(V')\\
&=&N^{p-r}c_p^r(U)c_r^p(V')
=N^{p-r}c_p^r(U)N^{r-p}c_p^r(V)\\
&=&c_p^r(U)c_p^r(V)
\end{eqnarray*}

Here we have used Theorem 1.9 (1). Now since this formula tells us that the moments of $\mu_W^r$ are bounded by those of $\mu^r_U\times\mu^r_V$, this gives the last assertion as well.
\end{proof}

\section{Abelian groups}

In this section we further restrict the attention, to a very special class of projective models. Generally speaking, the problem is that the complex Hadamard matrices, which are the main source of projective models, are quite complicated objects, and the only elementary example is the Fourier matrix $F_X$ of a finite abelian group $X$. See \cite{tzy}.

Let us first recall the construction of this matrix:

\begin{proposition}
Let $X=\mathbb Z_{N_1}\times\ldots\times\mathbb Z_{N_k}$ be a finite abelian group, and consider the matrix $F_X=F_{N_1}\otimes\ldots\otimes F_{N_k}$, where $F_N=(w^{ij})$, with $w=e^{2\pi i/N}$.
\begin{enumerate}
\item In the cyclic group case, $X=\mathbb Z_N$, we have $F_X=F_N$.
 
\item In general, $F_X$ is the matrix of the Fourier transform over $X$.

\item With $F=F_X$ we have $F_{i+j,k}=F_{ik}F_{jk}$, $F_{i,j+k}=F_{ij}F_{ik}$, $F_{-i,j}=F_{i,-j}=\bar{F}_{ij}$.
\end{enumerate}
\end{proposition} 

\begin{proof}
All these results are well-known:

(1) This is clear from definitions.

(2) This is well-known in the cyclic group case, and in general, it follows by using the compatibility between the product of groups $\times$ and the tensor product of matrices $\otimes$.

(3) This is clear in the cyclic group case, and then in general as well. 
\end{proof}

Observe that each $F_N$, and hence each $F_X$, is a complex Hadamard matrix.

Now let us go back to Proposition 2.3 above. By using the formulae in (3) above, we see that the matrix constructed there, with $H=F_X$, is given by:
$$(U_{ij})_{kl}=\frac{1}{N}\cdot\frac{F_{ik}F_{jl}}{F_{il}F_{jk}}=\frac{1}{N}(F_{ik}F_{i,-l})(F_{-j,k}F_{-j,-l})=\frac{1}{N}F_{i,k-l}F_{-j,k-l}=\frac{1}{N}F_{i-j,k-l}$$

Thus, the projective models associated to the Fourier matrices, coming from Proposition 2.3 above, can be in fact introduced directly, as follows:

\begin{definition}
Associated to a finite abelian group $X$ is the projective model 
$$\pi:C(X)\to M_{|X|}(\mathbb C)$$
coming from the matrix $(U_{ij})_{kl}=\frac{1}{N}F_{i-j,k-l}$, where $F=F_X$.
\end{definition}

Observe that the models $U,U'$ fall into the general framework of Proposition 1.8 (2) above, but with both $U,U'$ being twisted by the Fourier transform.

Now let $X,Y$ be finite abelian groups, and let us try to understand the projective model constructed by deforming the tensor product of the corresponded Fourier models:

\begin{definition}
Given two finite abelian groups $X,Y$, we consider the corresponding Fourier models $U,V$, we construct $W=U\otimes_QV$ as in Proposition 2.5, and we factorize
$$\xymatrix{C(S_{X\times Y}^+)\ar[rr]^{\pi_Q}\ar[rd]&&M_{X\times Y}(\mathbb C)\\&C(G_Q)\ar[ur]_\pi&}$$
with $C(G_Q)$ being the Hopf image of $\pi_Q$, as in Definition 1.2.
\end{definition}

Explicitely computing the compact quantum group $G_Q$, as function of the parameter matrix $Q\in M_{X\times Y}(\mathbb T)$, and understanding the random walk on the corresponding group dual $\Gamma_Q=\widehat{G}_Q$, will be our main purpose, in the reminder of this paper.

In order to do so, we use the following notion, from \cite{bic}:

\begin{definition}
Let $C(S_M^+)\to A$ and $C(S_N^+)\to B$ be Hopf algebra quotients, with fundamental corepresentations denoted $u,v$. We let
$$A*_wB=A^{*N}*B/<[u_{ab}^{(i)},v_{ij}]=0>$$
with the Hopf algebra structure making $w_{ia,jb}=u_{ab}^{(i)}v_{ij}$ a corepresentation.
\end{definition}

The fact that we have indeed a Hopf algebra follows from the fact that $w$ is magic. In terms of quantum groups, if $A=C(G)$, $B=C(H)$, we write $A*_wB=C(G\wr_*H)$:
$$C(G)*_wC(H)=C(G\wr_*H)$$

The $\wr_*$ operation is then the free analogue of $\wr$, the usual wreath product. See \cite{bic}.

We will need as well the following elementary lemma: 

\begin{lemma}
If $X$ is a finite abelian group then
$$C(X)=C(S_X^+)/<u_{ij}=u_{kl}|\forall i-j=k-l>$$
with all the indices taken inside $X$.
\end{lemma}

\begin{proof}
Observe first that $C(Y)=C(S_X^+)/<u_{ij}=u_{kl}|\forall i-j=k-l>$ is commutative, because $u_{ij}u_{kl}=u_{ij}u_{i,l-k+i}=\delta_{j,l-k+i}u_{ij}$ and $u_{kl}u_{ij}=u_{i,l-k+i}u_{ij}=\delta_{j,l-k+i}u_{ij}$. Thus we have $Y\subset S_X$, and since $u_{ij}(\sigma)=\delta_{i\sigma(j)}$ for any $\sigma\in Y$, we obtain:
$$i-j=k-l\implies(\sigma(j)=i\iff\sigma(l)=k)$$

But this condition tells us precisely that $\sigma(i)-i$ must be independent on $i$, and so $\sigma(i)=i+x$ for some $x\in X$, and so $\sigma\in X$, as desired.
\end{proof}

We can now factorize representation $\pi_Q$ in Definition 3.3, as follows:

\begin{proposition}
We have a factorization
$$\xymatrix{C(S_{X\times Y}^+)\ar[rr]^{\pi_Q}\ar[rd]&&M_{X\times Y}(\mathbb C)\\&C(Y\wr_*X)\ar[ur]_\pi&}$$
given by $U_{ab}^{(i)}=\sum_jW_{ia,jb}$ and by $V_{ij}=\sum_aW_{ia,jb}$, independently of $b$.
\end{proposition}

\begin{proof}
With $K=F_X,L=F_Y$ and $M=|X|,N=|Y|$, the formula of the magic matrix $W\in M_{X\times Y}(M_{X\times Y}(\mathbb C))$ associated to $H=K\otimes_QL$ is:
$$(W_{ia,jb})_{kc,ld}=\frac{1}{MN}\cdot\frac{Q_{ic}Q_{jd}}{Q_{id}Q_{jc}}\cdot\frac{K_{ik}K_{jl}}{K_{il}K_{jk}}\cdot\frac{L_{ac}L_{bd}}{L_{ad}L_{bc}}=\frac{1}{MN}\cdot\frac{Q_{ic}Q_{jd}}{Q_{id}Q_{jc}}\cdot K_{i-j,k-l}L_{a-b,c-d}$$

Our claim that the representation $\pi_Q$ constructed in Definition 3.3 can be factorized in three steps, up to the factorization in the statement, as follows:
$$\xymatrix{C(S_{X\times Y}^+)\ar[rr]^{\pi_Q}\ar[d]&&M_{X\times Y}(\mathbb C)\\C(S_Y^+\wr_*S_X^+)\ar[r]\ar@{.>}[rru]&C(S_Y^+\wr_*X)\ar[r]\ar@{.>}[ur]&C(Y\wr_*X)\ar@{.>}[u]}$$

Indeed, the construction of the map on the left is standard, see \cite{bic}, and this produces the first factorization. Regarding the second factorization, this comes from the fact that since the elements $V_{ij}$ depend on $i-j$, they satisfy the defining relations for the quotient algebra $C(S_X^+)\to C(X)$, coming from Lemma 3.5. Finally, regarding the third factorization, observe that the above matrix $W_{ia,jb}$ depends only on $a-b$. By summing over $j$ we obtain that $U_{ab}^{(i)}$ depends only on $a-b$, and by using Lemma 3.5, we are done.
\end{proof}

In order to further factorize the representation in Proposition 3.6, we use:

\begin{definition}
If $H\curvearrowright\Gamma$ is a finite group acting by automorphisms on a discrete group, the corresponding crossed coproduct Hopf algebra is
$$C^*(\Gamma)\rtimes C(H)=C^*(\Gamma)\otimes C(H)$$
with comultiplication $\Delta(r\otimes\delta_k)=\sum_{h\in H}(r\otimes\delta_h)\otimes(h^{-1}\cdot r\otimes\delta_{h^{-1}k})$, for $r\in\Gamma,k\in H$.
\end{definition}

Observe that $C(H)$ is a subcoalgebra, and that $C^*(\Gamma)$ is not a subcoalgebra. The quantum group corresponding to $C^*(\Gamma)\rtimes C(H)$ is denoted $\widehat{\Gamma}\rtimes H$.

Now back to the factorization in Proposition 3.6, the point is that we have:

\begin{lemma}
With $L=F_Y,N=|Y|$ we have an isomorphism
$$C(Y\wr_*X)\simeq C^*(Y)^{*X}\rtimes C(X)$$
given by $v_{ij}\to1\otimes v_{ij}$ and $u_{ab}^{(i)}=\frac{1}{N}\sum_cL_{b-a,c}c^{(i)}\otimes 1$.
\end{lemma}

\begin{proof}
We know that $C(Y\wr_*X)$ is the quotient of $C(Y)^{*X}*C(X)$ by the relations $[u_{ab}^{(i)},v_{ij}]=0$. Now since $v_{ij}$ depends only on $j-i$, we obtain $[u_{ab}^{(i)},v_{kl}]=[u_{ab}^{(i)},v_{i,l-k+i}]=0$, and so we are in a usual tensor product situation, and we have:
$$C(Y\wr_*X)=C(Y)^{*X}\otimes C(X)$$

Let us compose now this identification with $\Phi^{*X}\otimes id$, where $\Phi:C(Y)\to C^*(Y)$ is the Fourier transform. We obtain an isomorphism as in the statement, and since $\Phi(u_{ab})=\frac{1}{N}\sum_cL_{b-a,c}c$, the formula for the image of $u_{ab}^{(i)}$ is indeed the one in the statement.
\end{proof}

Here is now our key lemma, which will lead to further factorizations:

\begin{lemma}
With $c^{(i)}=\sum_aL_{ac}u_{a0}^{(i)}$ and $\varepsilon_{ke}=\sum_iK_{ik}e_{ie}$ we have:
$$\pi(c^{(i)})(\varepsilon_{ke})=\frac{Q_{i,e-c}Q_{i-k,e}}{Q_{ie}Q_{i-k,e-c}}\varepsilon_{k,e-c}$$
In particular if $c_1+\ldots+c_s=0$ then $\pi(c_1^{(i_1)}\ldots c_s^{(i_s)})$ is diagonal, for any $i_1,\ldots,i_s$.
\end{lemma}

\begin{proof}
We have the following formula:
$$\pi(c^{(i)})=\sum_aL_{ac}\pi(u_{a0}^{(i)})=\sum_{aj}L_{ac}W_{ia,j0}$$

On the other hand, in terms of the basis in the statement, we have:
$$W_{ia,jb}(\varepsilon_{ke})=\frac{1}{N}\delta_{i-j,k}\sum_d\frac{Q_{id}Q_{je}}{Q_{ie}Q_{jd}}L_{a-b,d-e}\varepsilon_{kd}$$

We therefore obtain, as desired:
\begin{eqnarray*}
\pi(c^{(i)})(\varepsilon_{ke})
&=&\frac{1}{N}\sum_{ad}L_{ac}\frac{Q_{id}Q_{i-k,e}}{Q_{ie}Q_{i-k,d}}L_{a,d-e}\varepsilon_{kd}=\frac{1}{N}\sum_d\frac{Q_{id}Q_{i-k,e}}{Q_{ie}Q_{i-k,d}}\varepsilon_{kd}\sum_aL_{a,d-e+c}\\
&=&\sum_d\frac{Q_{id}Q_{i-k,e}}{Q_{ie}Q_{i-k,d}}\varepsilon_{kd}\delta_{d,e-c}=\frac{Q_{i,e-c}Q_{i-k,e}}{Q_{ie}Q_{i-k,e-c}}\varepsilon_{k,e-c}
\end{eqnarray*}

Regarding now the last assertion, this follows from the fact that each matrix of type $\pi(c_r^{(i_r)})$ acts on the standard basis elements $\varepsilon_{ke}$ by preserving the left index $k$, and by rotating by $c_r$ the right index $e$. Thus when we assume $c_1+\ldots+c_s=0$ all these rotations compose up to the identity, and we obtain indeed a diagonal matrix.
\end{proof}

We have now all needed ingredients for refining Proposition 3.6:

\begin{theorem}
We have a factorization as follows,
$$\xymatrix{C(S_{X\times Y}^+)\ar[rr]^{\pi_Q}\ar[rd]&&M_{X\times Y}(\mathbb C)\\&C^*(\Gamma_{X,Y})\rtimes C(X)\ar[ur]_\rho&}$$
where $\Gamma_{X,Y}=Y^{*X}/<[c_1^{(i_1)}\ldots c_s^{(i_s)},d_1^{(j_1)}\ldots d_s^{(j_s)}]=1|\sum_rc_r=\sum_rd_r=0>$.
\end{theorem}

\begin{proof}
Assume that we have a representation $\pi:C^*(\Gamma)\rtimes C(X)\to M_L(\mathbb C)$, let $\Lambda$ be a $X$-stable normal subgroup of $\Gamma$, so that $X$ acts on $\Gamma/\Lambda$ and that we can form the crossed coproduct $C^*(\Gamma/\Lambda)\rtimes C(X)$, and assume that $\pi$ is trivial on $\Lambda$. Then $\pi$ factorizes as:
$$\xymatrix{C^*(\Gamma)\rtimes C(X)\ar[rr]^\pi\ar[rd]&&M_L(\mathbb C)\\&C^*(\Gamma/\Lambda)\rtimes C(X)\ar[ur]_\rho}$$

With $\Gamma=Y^{*X}$, and by using Lemma 3.8 and Lemma 3.9, this gives the result.
\end{proof}

\section{Formal deformations}

In general, further factorizing the representation found in Theorem 3.10 above is a quite complicated task. In this section we restrict attention to the case where the parameter matrix $Q$ is generic, in the sense that its entries are as algebrically independent as possible, and we prove that the representation in Theorem 3.10 is the minimal one.

Our starting point is the group $\Gamma_{X,Y}$ found above:

\begin{definition}
Associated to two finite abelian groups $X,Y$ is the discrete group
$$\Gamma_{X,Y}=Y^{*X}\Big/\left<[c_1^{(i_1)}\ldots c_s^{(i_s)},d_1^{(j_1)}\ldots d_s^{(j_s)}]=1\Big|\sum_rc_r=\sum_rd_r=0\right>$$
where the superscripts refer to the $X$ copies of $Y$, inside the free product.
\end{definition}

We will need a more convenient description of this group. The idea here is that the above commutation relations can be realized inside a suitable semidirect product.

Given a group acting on another group, $H\curvearrowright G$, we denote as usual by $G\rtimes H$ the semidirect product of $G$ by $H$, i.e. the set $G\times H$, with multiplication $(a,s)(b,t)=(as(b),st)$. Now given a group $G$, and a finite abelian group $Y$, we can make $Y$ act on $G^Y$, and form the product $G^Y\rtimes Y$. Since the elements of type $(g,\ldots,g)$ are invariant, we can form as well the product $(G^Y/G)\rtimes Y$, and by identifying $G^Y/G\simeq G^{|Y|-1}$ via the map $(1,g_1,\ldots,g_{|Y|-1})\to(g_1,\ldots,g_{|Y|-1})$, we obtain a product $G^{|Y|-1}\rtimes Y$. 

With these notations, we have the following result:

\begin{proposition}
The group $\Gamma_{X,Y}$ has the following properties:
\begin{enumerate}
\item $\Gamma_{X,Y}\simeq\mathbb Z^{(|X|-1)(|Y|-1)}\rtimes Y$.
\item $\Gamma_{X,Y}\subset\mathbb Z^{(|X|-1)|Y|}\rtimes Y$ via $c^{(0)}\to(0,c)$ and $c^{(i)}\to(b_{i0}-b_{ic},c)$ for $i\neq 0$, where $b_{ic}$ are the standard generators of $\mathbb Z^{(|X|-1)|Y|}$.
\end{enumerate}
\end{proposition}

\begin{proof}
We prove these assertions at the same time. We must prove that we have group morphisms, given by the formulae in the statement, as follows:
$$\Gamma_{X,Y}\simeq\mathbb Z^{(|X|-1)(|Y|-1)}\rtimes Y\subset \mathbb Z^{(|X|-1)|Y|}\rtimes Y$$

Our first claim is that the formula in (2) defines a morphism $\Gamma_{X,Y}\to\mathbb Z^{(|X|-1)|Y|}\rtimes Y$. Indeed, the elements $(0,c)$ produce a copy of $Y$, and since we have a group embedding $Y\subset\mathbb Z^{|Y|}\rtimes Y$ given by $c\to(b_0-b_c,c)$, the elements $C^{(i)}=(b_{i0}-b_{ic},c)$ produce a copy of $Y$, for any $i\neq 0$. In order to check now the commutation relations, observe that:
$$C_1^{(i_1)}\ldots C_s^{(i_s)}=\left(b_{i_10}-b_{i_1c_1}+b_{i_2c_1}-b_{i_2,c_1+c_2}+\ldots+b_{i_s,c_1+\ldots+c_{s-1}}-b_{i_s,c_1+\ldots+c_s},\sum_rc_r\right)$$

Thus $\sum_rc_r=0$ implies $C_1^{(i_1)}\ldots C_s^{(i_s)}\in\mathbb Z^{(|X|-1)|Y|}$, and since we are now inside an abelian group, we have the commutation relations, and our claim is proved.

Using the considerations before the statement of the proposition, it is routine to construct an embedding $\mathbb Z^{(|X|-1)(|Y|-1)}\rtimes Y\subset \mathbb Z^{(|X|-1)|Y|}\rtimes Y$ such that we have group morphisms whose composition is the group morphism just constructed, as follows:
$$\Gamma_{X,Y}\to\mathbb Z^{(|X|-1)(|Y|-1)}\rtimes Y\subset \mathbb Z^{(|X|-1)|Y|}\rtimes Y$$

It remains to prove that the map on the left is injective. For this purpose, consider the morphism $\Gamma_{X,Y}\to Y$ given by $c^{(i)}\to c$, whose kernel $T$ is formed by the elements of type $c_1^{(i_1)} \ldots c_s^{(i_s)}$, with $\sum_rc_r=0$. We get an exact sequence, as follows:
$$1\to T\to\Gamma_{X,Y}\to Y\to1$$

This sequence splits by $c\to c^{(0)}$, so we have $\Gamma_{X,Y}\simeq T \rtimes Y$. Now by the definition of $\Gamma_{X,Y}$, the subgroup $T$ constructed above is abelian, and is moreover generated by the elements
$(-c)^{(0)}c^{(i)}$, $i,c \not=0$.  Finally, the fact that $T$ is freely generated by these elements follows from the computation in the proof of Lemma 4.4 below.
\end{proof}

Let us specify now what our genericity assumptions are:

\begin{definition}
We use the following notions:
\begin{enumerate}
\item We call $p_1,\ldots,p_m\in\mathbb T$ root independent if for any $r_1,\ldots, r_m\in\mathbb Z$ we have $p_1^{r_1}\ldots p_m^{r_m}=1\implies r_1=\ldots=r_m=0$.

\item A matrix $Q\in M_{X\times Y}(\mathbb T)$, taken to be dephased ($Q_{0c}=Q_{i0}=1$), is called generic if  
the elements $Q_{ic}$, with $i,c\neq0$, are root independent.
\end{enumerate}
\end{definition}

We will need the following lemma:

\begin{lemma}
Assume that $Q\in M_{X\times Y}(\mathbb T)$ is generic, and put
$$\theta_{ic}^{ke}=\frac{Q_{i,e-c}Q_{i-k,e}}{Q_{ie}Q_{i-k,e-c}}$$
For every $k \in X$, we have a representation  $\pi^k : \Gamma_{X,Y}\rightarrow U_{|Y|}$ given by $\pi^k(c^{(i)})\epsilon_e=\theta_{ic}^{ke}\epsilon_{e-c}$. The family of representations $(\pi^k)_{k \in X}$ is projectively faithful in the sense that if for some $t \in \Gamma_{X,Y}$, we have that $\pi^k(t)$ is a scalar matrix for any $k$, then $t=1$.
\end{lemma}

\begin{proof}
The representations $\pi^k$ arise from Lemma 3.9. With $\Gamma_{X,Y}=T\rtimes Y$, as in the proof of Proposition 4.2, we see that for $t \in \Gamma_{X,Y}$ such that $\pi^k(t)$ is a scalar matrix for any $k$, then $t \in T$, since the elements of $T$ are the only ones having their image by $\pi^k$  formed by diagonal matrices. Now write $t=\prod_{i \not=0, c\not=0} ((-c)^{(0)}(c)^{(i)})^{R_{ic}}$ with the generators of $T$ as in the proof of Proposition 4.2, for $R_{ic}\in\mathbb Z$, and consider the quantities:
\begin{eqnarray*}
A(k,e)&=&\prod_{i\neq0}\prod_{c\neq0}(\theta_{ic}^{ke}(\theta_{0c}^{ke})^{^{-1}})^{R_{ic}}
=\prod_{i\neq0}\prod_{c\neq0} (\theta_{ic}^{ke})^{R_{ic}}(\theta_{0c}^{ke})^{-R_{ic}}\\
&=&\prod_{i\neq0}\prod_{c\neq0}(\theta_{ic}^{ke})^{R_{ic}}
\cdot\prod_{c\neq0}(\theta_{0c}^{ke})^{-\sum_{i\neq0}R_{ic}}
=\prod_{j\neq0}\prod_{c\neq0} (\theta_{jc}^{ke})^{R_{jc}}
\cdot\prod_{c\neq0}\prod_{j\neq0}(\theta_{jc}^{ke})^{\sum_{i\neq0}R_{ic}}\\
&=&\prod_{j\neq0}\prod_{c\neq0}(\theta_{jc}^{ke})^{R_{jc}+\sum_{i\neq0}R_{ic}}
\end{eqnarray*}

We have $\pi^k(t)(\epsilon_e)= A(k,e)\epsilon_e$ for any $k,e$. 
Our assumption is that for any $k$, we have $A(k,e)=A(k,f)$ for any $e,f$. Using the root independence of the elements $Q_{ic}$, $i,c \not=0$, we see that this implies $R_{ic}=0$ for any $i,c$, and this proves our assertion.  
\end{proof}

We will need as well the following lemma:

\begin{lemma}
Let $\pi:C^*(\Gamma)\rtimes  C(H) \rightarrow L$ be a surjective Hopf algebra map, such that $\pi_{|C(H)}$ is injective, and such that for $r \in \Gamma$ and $f \in C(H)$, we have:
$$\pi(r \otimes 1)=\pi(1 \otimes f) \implies r=1$$ 
Then $\pi$ is an isomorphism.
\end{lemma}

\begin{proof} 
We use here various tools from \cite{ade}, \cite{sch}. Put $A=C^*(\Gamma)\rtimes C(H)$. We start with the following Hopf algebra exact sequence, where $i(f)=1\otimes f$ and $p=\varepsilon\otimes 1$:
$$\mathbb C\to C(H)\overset{i}\to A \overset{p}\to C^*(\Gamma)\to
\mathbb C$$ 

Since $\pi\circ i$ is injective, and Hopf subalgebra $\pi\circ i(C(H))$ is central in $L$, we can form the quotient Hopf algebra $\overline{L} = L/ (\pi\circ i(C(H))^+L$, and we get another exact sequence:
$$\mathbb C\to C(H)\xrightarrow{\pi \circ i} L \overset{q}\to \overline{L}  \to
\mathbb C$$ 

Note that this sequence is indeed exact, e.g. by centrality. So we get the following diagram with exact rows, with the Hopf algebra map on the right surjective:
$$\xymatrix{\mathbb C\ar[r]&C(H)\ar@2@{-}[d]\ar[r]^i&A\ar[d]^\pi\ar[r]^p&C^*(\Gamma)\ar[r]\ar[d]&\mathbb C\\
\mathbb C\ar[r]&C(H)\ar[r]^{\pi\circ i}&L\ar[r]^q&\overline{L}\ar[r]&\mathbb C}$$

Since a quotient of a group algebra is still a group algebra, we get a commutative diagram with exact rows as follows:
$$\xymatrix{\mathbb C\ar[r]&C(H)\ar@2@{-}[d]\ar[r]^i&A\ar[d]^\pi\ar[r]^p&C^*(\Gamma)\ar[r]\ar[d]&\mathbb C\\
\mathbb C\ar[r]&C(H)\ar[r]^{\pi\circ i}&L\ar[r]^{q'}&C^*(\overline{\Gamma})\ar[r]&\mathbb C}$$

Here the Hopf algebra map on the right is induced by a surjective morphism $u : \Gamma \rightarrow \overline{\Gamma}$, $g \mapsto \overline{g}$. By the five lemma we just have to show that $u$ is injective. So, let $g \in \Gamma$ be such that $u(g)=1$. Then $q' \pi(g \otimes 1) = u p(g\otimes 1)=u(g)=\overline{g}=1$. For $g \in \Gamma$, put:
$$_gA= \{a \in A \ | \ p(a_1) \otimes a_2= g \otimes a\}$$
$$_{\overline{g}}L= \{l \in L \ | \ q'(l_1) \otimes l_2= \overline{g} \otimes l\}$$

The commutativity of the right square ensures that $\pi(_gA) \subset {_{\overline{g}}L}$.
Then with the previous $g$, we have $\pi(g \otimes 1) \in {_{\overline{1}}L} = \pi i (C(H))$ (exactness of the sequence), so $\pi(g \otimes 1)= \pi(1 \otimes f)$
for some $f \in C(H)$. We conclude by our assumption that $g=1$.
\end{proof}

We have now all ingredients for proving our first main result:

\begin{theorem}
When $Q$ is generic, the minimal factorization for $\pi_Q$ is 
$$\xymatrix{C(S_{X\times Y}^+)\ar[rr]^{\pi_Q}\ar[rd]&&M_{X\times Y}(\mathbb C)\\&C^*(\Gamma_{X,Y})\rtimes C(X)\ar[ur]_\pi&}$$
where $\Gamma_{X,Y}\simeq\mathbb Z^{(|X|-1)(|Y|-1)}\rtimes Y$ is the discrete group constructed above.
\end{theorem}

\begin{proof}
We want to apply Lemma 4.5 to the morphism $\theta : C^*(\Gamma_{X,Y})\rtimes C(X)\to L$ arising from the factorization in Theorem 3.10, where $L$ denotes the Hopf image of $\pi_Q$, which produces the following commutative diagram (see \cite{bb2}):
$$\xymatrix{
C(S_{X \times Y}^+) \ar[rr]^{\pi_Q} \ar[dr]_{} \ar@/_/[ddr]_{}& & M_{X \times Y}(\mathbb C) \\
& L \ar[ur]_{} & \\
& C^*(\Gamma_{X,Y})\rtimes C(X) \ar@{-->}[u]_\theta \ar@/_/[uur]_{\pi}& 
}$$

The first observation is that the injectivity assumption on $C(X)$ holds by construction, and that
for $f \in C(X)$, the matrix $\pi(f)$ is ``block scalar'', the blocks corresponding to the indices $k$ in the basis $\varepsilon_{ke}$ in the basis from Lemma 3.9.
Now for $r \in \Gamma_{X,Y}$ with $\theta(r\otimes 1)=\theta(1 \otimes f)$ for some $f \in C(X)$, we see, using the commutative diagram, that we will have that 
$\pi(r \otimes 1)$ is block scalar. By Lemma 4.4, the family of representations $(\pi^k)$ of $\Gamma_{X,Y}$, corresponding to the blocks $k$,  is projectively faithful, so $r=1$.
  We can apply indeed Lemma 4.5, and we are done.
\end{proof}

\section{Random walks}

In this section we compute the Kesten type measure $\mu=law(\chi)$ for the quantum group $G=G_Q$ found in Theorem 4.6. Our results here will be a combinatorial moment formula, a geometric interpretation of it, and an asymptotic convergence result.

The moment formula is as follows:

\begin{proposition}
We have the moment formula
$$\int_G\chi^p
=\frac{1}{|X|\cdot|Y|}\#\left\{\begin{matrix}i_1,\ldots,i_p\in X\\ d_1,\ldots,d_p\in Y\end{matrix}\Big|\begin{matrix}[(i_1,d_1),(i_2,d_2),\ldots,(i_p,d_p)]\ \ \ \ \\=[(i_1,d_p),(i_2,d_1),\ldots,(i_p,d_{p-1})]\end{matrix}\right\}$$
where the sets between square brackets are by definition sets with repetition.
\end{proposition}

\begin{proof}
According to the various formulae in sections 2 and 3 above, the factorization found in Theorem 4.6 is, at the level of standard generators, as follows:
$$\begin{matrix}
C(S_{X\times Y}^+)&\to&C^*(\Gamma_{X,Y})\otimes C(X)&\to&M_{X\times Y}(\mathbb C)\\
u_{ia,jb}&\to&\frac{1}{|Y|}\sum_cF_{b-a,c}c^{(i)}\otimes v_{ij}&\to&W_{ia,jb}
\end{matrix}$$

Thus, the main character is given by:
$$\chi=\frac{1}{|Y|}\sum_{iac}c^{(i)}\otimes v_{ii}=\sum_{ic}c^{(i)}\otimes v_{ii}=\left(\sum_{ic}c^{(i)}\right)\otimes\delta_1$$

Now since the Haar functional of $C^*(\Gamma)\rtimes C(H)$ is the tensor product of the Haar functionals of $C^*(\Gamma),C(H)$, this gives the following formula, valid for any $p\geq1$:
$$\int_G\chi^p=\frac{1}{|X|}\int_{\widehat{\Gamma}_{X,Y}}\left(\sum_{ic}c^{(i)}\right)^p$$

Let $S_i=\sum_cc^{(i)}$. By using the embedding in Proposition 4.2 (2), with the notations there we have $S_i=\sum_c(b_{i0}-b_{ic},c)$, and these elements multiply as follows:
$$S_{i_1}\ldots S_{i_p}=\sum_{c_1\ldots c_p}
\begin{pmatrix}
b_{i_10}-b_{i_1c_1}+b_{i_2c_1}-b_{i_2,c_1+c_2}&&\\
+b_{i_3,c_1+c_2}-b_{i_3,c_1+c_2+c_3}+\ldots\ldots&,&c_1+\ldots+c_p&\\
\ldots\ldots+b_{i_p,c_1+\ldots+c_{p-1}}-b_{i_p,c_1+\ldots+c_p}&&
\end{pmatrix}$$

In terms of the new indices $d_r=c_1+\ldots+c_r$, this formula becomes:
$$S_{i_1}\ldots S_{i_p}=\sum_{d_1\ldots d_p}
\begin{pmatrix}
b_{i_10}-b_{i_1d_1}+b_{i_2d_1}-b_{i_2d_2}&&\\
+b_{i_3d_2}-b_{i_3d_3}+\ldots\ldots&,&d_p&\\
\ldots\ldots+b_{i_pd_{p-1}}-b_{i_pd_p}&&
\end{pmatrix}$$

Now by integrating, we must have $d_p=0$ on one hand, and on the other hand:
$$[(i_1,0),(i_2,d_1),\ldots,(i_p,d_{p-1})]=[(i_1,d_1),(i_2,d_2),\ldots,(i_p,d_p)]$$

Equivalently, we must have $d_p=0$ on one hand, and on the other hand:
$$[(i_1,d_p),(i_2,d_1),\ldots,(i_p,d_{p-1})]=[(i_1,d_1),(i_2,d_2),\ldots,(i_p,d_p)]$$

Thus, by translation invariance with respect to $d_p$, we obtain:
$$\int_{\widehat{\Gamma}_{X,Y}}S_{i_1}\ldots S_{i_p}
=\frac{1}{|Y|}\#\left\{d_1,\ldots,d_p\in Y\Big|\begin{matrix}[(i_1,d_1),(i_2,d_2),\ldots,(i_p,d_p)]\ \ \ \ \\=[(i_1,d_p),(i_2,d_1),\ldots,(i_p,d_{p-1})]\end{matrix}\right\}$$

It follows that we have the following moment formula:
$$\int_{\widehat{\Gamma}_{X,Y}}\left(\sum_iS_i\right)^p
=\frac{1}{|Y|}\#\left\{\begin{matrix}i_1,\ldots,i_p\in X\\ d_1,\ldots,d_p\in Y\end{matrix}\Big|\begin{matrix}[(i_1,d_1),(i_2,d_2),\ldots,(i_p,d_p)]\ \ \ \ \\=[(i_1,d_p),(i_2,d_1),\ldots,(i_p,d_{p-1})]\end{matrix}\right\}$$

Now by dividing by $|X|$, we obtain the formula in the statement.
\end{proof} 

The formula in Proposition 5.1 can be interpreted as follows:

\begin{proposition}
With $M=|X|,N=|Y|$ we have the formula
$$law(\chi)=\left(1-\frac{1}{N}\right)\delta_0+\frac{1}{N}law(A)$$
where $A\in C(\mathbb T^{MN},M_M(\mathbb C))$ is given by $A(q)=$ Gram matrix of the rows of $q$.
\end{proposition}

\begin{proof}
According to Proposition 5.1, we have the following formula:
\begin{eqnarray*}
\int_G\chi^p
&=&\frac{1}{MN}\sum_{i_1\ldots i_p}\sum_{d_1\ldots d_p}\delta_{[i_1d_1,\ldots,i_pd_p],[i_1d_p,\ldots,i_pd_{p-1}]}\\
&=&\frac{1}{MN}\int_{\mathbb T^{MN}}\sum_{i_1\ldots i_p}\sum_{d_1\ldots d_p}\frac{q_{i_1d_1}\ldots q_{i_pd_p}}{q_{i_1d_p}\ldots q_{i_pd_{p-1}}}\,dq\\
&=&\frac{1}{MN}\int_{\mathbb T^{MN}}\sum_{i_1\ldots i_p}\left(\sum_{d_1}\frac{q_{i_1d_1}}{q_{i_2d_1}}\right)\left(\sum_{d_2}\frac{q_{i_2d_2}}{q_{i_3d_2}}\right)\ldots\left(\sum_{d_p}\frac{q_{i_pd_p}}{q_{i_1d_p}}\right)dq
\end{eqnarray*}

Consider now the Gram matrix in the statement, $A(q)_{ij}=<R_i,R_j>$, where $R_1,\ldots,R_M$ are the rows of $q\in \mathbb T^{MN}\simeq M_{M\times N}(\mathbb T)$. We have then:
\begin{eqnarray*}
\int_G\chi^p
&=&\frac{1}{MN}\int_{\mathbb T^{MN}}<R_{i_1},R_{i_2}><R_{i_2},R_{i_3}>\ldots<R_{i_p},R_{i_1}>\\
&=&\frac{1}{MN}\int_{\mathbb T^{MN}}A(q)_{i_1i_2}A(q)_{i_2i_3}\ldots A(q)_{i_pi_1}\\
&=&\frac{1}{MN}\int_{\mathbb T^{MN}}Tr(A(q)^p)dq=\frac{1}{N}\int_{\mathbb T^{MN}}tr(A(q)^p)dq
\end{eqnarray*}

But this gives the formula in the statement, and we are done.
\end{proof}

The problem now is that of finding the good regime, $M=f(K),N=g(K),K\to\infty$, where the measure in Proposition 5.1 converges, after some suitable manipulations.

We denote by $NC(p)$ the set of noncrossing partitions of $\{1,\ldots,p\}$, and for $\pi\in P(p)$ we denote by $|\pi|\in\{1,\ldots,p\}$ the number of blocks. See \cite{nsp}. We will need:

\begin{lemma}
With $M=\alpha K,N=\beta K$, $K\to\infty$ we have:
$$\frac{c_p}{K^{p-1}}\simeq\sum_{r=1}^p\#\left\{\pi\in NC(p)\Big||\pi|=r\right\}\alpha^{r-1}\beta^{p-r}$$
In particular, with $\alpha=\beta$ we have $c_p\simeq\frac{1}{p+1}\binom{2p}{p}(\alpha K)^{p-1}$.
\end{lemma}

\begin{proof}
We use the combinatorial formula in Proposition 5.1 above. Our claim is that, with $\pi=\ker(i_1,\ldots,i_p)$, for $\pi\in NC(p)$ the contribution to $c_p$ is $C_\pi\simeq\alpha^{|\pi|-1}\beta^{p-|\pi|}K^{p-1}$, and for $\pi\notin NC(p)$, the contribution is $C_\pi=O(K^{p-2})$.

As a first observation, since there are $M(M-1)\ldots (M-|\pi|+1)\simeq M^{|\pi|}$ choices for a multi-index $(i_1,\ldots,i_p)\in X^p$ satisfying $\ker i=\pi$, we have:
$$C_\pi\simeq M^{|\pi|-1}N^{-1}\#\left\{d_1,\ldots,d_p\in Y\Big|[d_\alpha|\alpha\in b]=[d_{\alpha-1}|\alpha\in b],\forall b\in\pi\right\}$$

Consider now the partition $\sigma=\ker d$. The contribution of $\sigma$ to the above quantity $C_\pi$ is then given by $\Delta(\pi,\sigma)N(N-1)\ldots(N-|\sigma|+1)\simeq\Delta(\pi,\sigma)N^{|\sigma|}$, where:
$$\Delta(\pi,\sigma)=\begin{cases}
1&{\rm if}\ |b\cap c|=|(b-1)\cap c|,\forall b\in\pi,\forall c\in\sigma\\
0&{\rm otherwise}
\end{cases}$$

We use now the fact, coming from \cite{bia}, that for $\pi,\sigma\in P(p)$ satisfying $\Delta(\pi,\sigma)=1$ we have $|\pi|+|\sigma|\leq p+1$, with equality when $\pi,\sigma\in NC(p)$ are inverse to each other, via Kreweras complementation. This shows that for $\pi\notin NC(p)$ we have $C_\pi=O(K^{p-2})$, and that for $\pi\in NC(p)$ we have $C_\pi\simeq M^{|\pi|-1}N^{-1}N^{p-|\pi|-1}=\alpha^{|\pi|-1}\beta^{p-|\pi|}K^{p-1}$, as claimed.
\end{proof}

We denote by $\pi_t$ the free Poisson (or Marchenko-Pastur) law of parameter $t>0$. It is known that the $p$-th moment of $\pi_t$ is given by $\sum_{\pi\in NC(p)}t^{|\pi|}$, and in particular that the $p$-th moment of $\pi_1$ is the Catalan number $\frac{1}{p+1}\binom{2p}{p}$. See \cite{mpa}, \cite{nsp}, \cite{vdn}.

Also, we denote by $D$ the dilation operation, $D_r(law(X))=law(rX)$.

\begin{theorem}
With $M=\alpha K,N=\beta K$, $K\to\infty$ we have:
$$\mu=\left(1-\frac{1}{\alpha\beta K^2}\right)\delta_0+\frac{1}{\alpha\beta K^2}D_{\frac{1}{\beta K}}(\pi_{\alpha/\beta})$$
In particular with $\alpha=\beta$ we have $\mu=\left(1-\frac{1}{\alpha^2K^2}\right)\delta_0+\frac{1}{\alpha^2K^2}D_{\frac{1}{\alpha K}}(\pi_1)$.
\end{theorem}

\begin{proof}
At $\alpha=\beta$, this follows from Lemma 5.3. In general now, we have:
$$\frac{c_p}{K^{p-1}}
\simeq\sum_{\pi\in NC(p)}\alpha^{|\pi|-1}\beta^{p-|\pi|}=\frac{\beta^p}{\alpha}\sum_{\pi\in NC(p)}\left(\frac{\alpha}{\beta}\right)^{|\pi|}=\frac{\beta^p}{\alpha}\int x^pd\pi_{\alpha/\beta}(x)$$

When $\alpha\geq\beta$, where $d\pi_{\alpha/\beta}(x)=\varphi_{\alpha/\beta}(x)dx$ is continuous, we obtain:
$$c_p
=\frac{1}{\alpha K}\int(\beta Kx)^p\varphi_{\alpha/\beta}(x)dx=\frac{1}{\alpha\beta K^2}\int x^p\varphi_{\alpha/\beta}\left(\frac{x}{\beta K}\right)dx$$

But this gives the formula in the statement. When $\alpha\leq\beta$ the computation is similar, with a Dirac mass as 0 dissapearing and reappearing, and gives the same result.
\end{proof}

As a first comment, when interchanging $\alpha,\beta$ we obtain $D_{\frac{1}{\beta K}}(\pi_{\alpha/\beta})=D_{\frac{1}{\alpha K}}(\pi_{\beta/\alpha})$, which is a consequence of the well-known formula $\pi_{t^{-1}}=D_t(\pi_t)$. This latter formula is best understood by using Kreweras complementation (see \cite{nsp}), which gives indeed:
$$\int x^pd\pi_t(x)=\sum_{\pi\in NC(p)}t^{|\pi|}=t^{p+1}\sum_{\pi\in NC(p)}t^{-|\pi|}=t\int(tx)^pd\pi_{t^{-1}}(x)$$

Let us state as well an explicit result, regarding densities:

\begin{proposition}
With $M=\alpha K,N=\beta K$, $K\to\infty$ we have:
$$\mu=\left(1-\frac{1}{\alpha\beta K^2}\right)\delta_0+\frac{1}{\alpha\beta K^2}\cdot\frac{\sqrt{4\alpha\beta K^2-(x-\alpha K-\beta K)^2}}{2\pi x}\,dx$$
In particular with $\alpha=\beta$ we have $\mu=\left(1-\frac{1}{\alpha^2K^2}\right)\delta_0+\frac{1}{\alpha^2K^2}\cdot\frac{\sqrt{\frac{4\alpha K}{x}-1}}{2\pi}$.
\end{proposition}

\begin{proof}
According to the well-known formula for the density of the free Poisson law (see \cite{mpa}, \cite{nsp}), the density of the continuous part $D_{\frac{1}{\beta K}}(\pi_{\alpha/\beta})$ is indeed given by:
$$\frac{\sqrt{4\frac{\alpha}{\beta}-(\frac{x}{\beta K}-1-\frac{\alpha}{\beta})^2}}
{2\pi\cdot\frac{x}{\beta K}}=\frac{\sqrt{4\alpha\beta K^2-(x-\alpha K-\beta K)^2}}{2\pi x}$$

With $\alpha=\beta$ now, we obtain the second formula in the statement, and we are done.
\end{proof}

Observe that at $\alpha=\beta=1$, where $M=N=K\to\infty$, the measure in Theorem 5.4, namely $\mu=\left(1-\frac{1}{K^2}\right)\delta_0+\frac{1}{K^2}D_{\frac{1}{K}}(\pi_1)$, is supported by $[0,4K]$. On the other hand, since the groups $\Gamma_{M,N}$ are all amenable, the corresponding measures are supported on $[0,MN]$, and so on $[0,K^2]$ in the $M=N=K$ situation. The fact that we don't have a convergence of supports is not surprising, because our convergence is in moments.


\begin{thebibliography}{99}

\bibitem{abi}N. Andruskiewitsch and J. Bichon, Examples of inner linear Hopf algebras, {\em Rev. Un. Mat. Argentina} {\bf 51} (2010), 7--18.

\bibitem{ade}N. Andruskiewitsch and J. Devoto, Extensions of Hopf algebras, {\em St. Petersburg Math. J.}  {\bf 7} (1996), 17--52.

\bibitem{ban}T. Banica, Truncation and duality results for Hopf image algebras, {\em Bull. Pol. Acad. Sci. Math.} {\bf 62} (2014), 161--179.

\bibitem{bb1}T. Banica and J. Bichon, Quantum groups acting on $4$ points, {\em J. Reine Angew. Math.} {\bf 626} (2009), 74--114.

\bibitem{bb2}T. Banica and J. Bichon, Hopf images and inner faithful representations, {\em Glasg. Math. J.} {\bf 52} (2010), 677--703.

\bibitem{bco}T. Banica and B. Collins, Integration over the Pauli quantum group, {\em J. Geom. Phys.} {\bf 58} (2008), 942--961.

\bibitem{bfs}T. Banica, U. Franz and A. Skalski, Idempotent states and the inner linearity property, {\em Bull. Pol. Acad. Sci. Math.} {\bf 60} (2012), 123--132. 

\bibitem{bsp}T. Banica and R. Speicher, Liberation of orthogonal Lie groups, {\em Adv. Math.} {\bf 222} (2009), 1461--1501.

\bibitem{bdd}J. Bhowmick, F. D'Andrea and L. Dabrowski, Quantum isometries of the finite noncommutative geometry of the standard model, {\em Comm. Math. Phys.} {\bf 307} (2011), 101--131.

\bibitem{bia}P. Biane, Some properties of crossings and partitions, {\em Discrete Math.} {\bf 175} (1997), 41--53.

\bibitem{bic}J. Bichon, Free wreath product by the quantum permutation group, {\em Alg. Rep. Theory} {\bf 7} (2004), 343--362.

\bibitem{bra}R. Brauer, On algebras which are connected with the semisimple continuous groups, {\em Ann. of Math.} {\bf 38} (1937), 857--872.

\bibitem{chi}A. Chirvasitu, Residually finite quantum group algebras, preprint 2014.

\bibitem{con}A. Connes, A unitary invariant in Riemannian geometry, {\em Int. J. Geom. Methods Mod. Phys.} {\bf 5} (2008), 1215--1242.

\bibitem{dri}V.G. Drinfeld, Quantum groups, Proc. ICM Berkeley (1986), 798--820.

\bibitem{end}D. Enders, A characterization of semiprojectivity for subhomogeneous $C^*$-algebras, preprint 2014.

\bibitem{fre}A. Freslon, On the partition approach to Schur-Weyl duality and free quantum groups, preprint 2014.

\bibitem{jim}M. Jimbo, A $q$-difference analog of $U(g)$ and the Yang-Baxter equation, {\em Lett. Math. Phys.} {\bf 10} (1985), 63--69.

\bibitem{kes}H. Kesten, Symmetric random walks on groups, {\em Trans. Amer. Math. Soc.} {\bf 92} (1959), 336--354.

\bibitem{mpa}V.A. Marchenko and L.A. Pastur, Distribution of eigenvalues in certain sets of random matrices, {\em Mat. Sb.} {\bf 72} (1967), 507--536.

\bibitem{ntu}S. Neshveyev and L. Tuset, Compact quantum groups and their representation categories, SMF (2013).

\bibitem{nsp}A. Nica and R. Speicher, Lectures on the combinatorics of free probability, Cambridge Univ. Press (2006).

\bibitem{rwe}S. Raum and M. Weber, The full classification of orthogonal easy quantum groups, preprint 2013.

\bibitem{sch}H.-J. Schneider, Some remarks on exact sequences of quantum groups, {\em Comm. Algebra} \textbf{21} (1993), 3337--3357.

\bibitem{tzy}W. Tadej and K. \.Zyczkowski, A concise guide to complex Hadamard matrices, {\em Open Syst. Inf. Dyn.} {\bf 13} (2006), 133--177.

\bibitem{vdn}D.V. Voiculescu, K.J. Dykema and A. Nica, Free random variables, AMS (1992).

\bibitem{wa1}S. Wang, Free products of compact quantum groups, {\em Comm. Math. Phys.} {\bf 167} (1995), 671--692.

\bibitem{wa2}S. Wang, Quantum symmetry groups of finite spaces, {\em Comm. Math. Phys.} {\bf 195} (1998), 195--211.

\bibitem{wo1}S.L. Woronowicz, Compact matrix pseudogroups, {\em Comm. Math. Phys.} {\bf 111} (1987), 613--665.

\bibitem{wo2}S.L. Woronowicz, Tannaka-Krein duality for compact matrix pseudogroups. Twisted SU(N) groups, {\em Invent. Math.} {\bf 93} (1988), 35--76.

\end{thebibliography}
\end{document}